\documentclass[11pt]{article}

\usepackage{amsmath} 
\usepackage{amssymb}
\usepackage{amsthm}
\usepackage{graphicx}
\usepackage{amscd}
\usepackage{epic, eepic}
\usepackage{url}
\usepackage{xcolor}
\usepackage{todonotes}
\usepackage{enumerate}
\usepackage{enumitem}
\usepackage{comment}
\usepackage{hyperref}
\usepackage{bbold}
\usepackage[utf8]{inputenc}
\newtheorem{theorem}{\bf Theorem}[section]
\newtheorem{lemma}[theorem]{\bf Lemma}

\newtheorem{prop}[theorem]{\bf Proposition}
\newtheorem{conj}[theorem]{\bf Conjecture}

\newtheorem{question}[theorem]{\bf Question}
\theoremstyle{definition}

\newtheorem{defi}[theorem]{\bf Definition}

\newcommand{\AG}{\mathrm{AG}}

\newcommand{\Z}{\mathbb{Z}}

\newcommand{\PG}{\mathrm{PG}}

\newcommand{\F}{\mathbb{F}}

\newcommand{\1}{\mathbb{1}}

\newcommand{\Span}{\operatorname{Span}}

          \usepackage{array}

\usepackage{todonotes}

\title{Cylinder type and $p$-divisible sets in $\mathbb{F}_p^3$}
\author{Gergely Kiss\thanks{Corvinus University of Budapest, Department of Mathematics,
Budapest, Hungary, and HUN-REN Alfr\'ed R\'enyi Institute of Mathematics, Budapest, Hungary. The author is supported by National Research, Development and Innovation Fund, OTKA grants no. FK 142993 and Starting 150576.
E-mail: {\tt kigergo57@gmail.com}} %\todo{affiliáció/mail/támogat}}
\and Ádám Markó \thanks{E\"otv\"os Lor\'and University, Budapest, Hungary.
The author is supported by OTKA grants no. FK 142993 and Starting 150576.
E-mail: {\tt marqadam@gmail.com}} %\todo{affiliáció/mail/támogató}
\and Zolt\'an L\'or\'ant Nagy\thanks{ELTE Linear Hypergraphs  Research Group,
 E\"otv\"os Lor\'and University, Budapest, Hungary. The author is supported by the J\'anos Bolyai Scholarship of the Hungarian Academy of Sciences and by the NRDI EXCELLENCE-24 grant no. 151504 Combinatorics and Geometry. %Hungarian Research Grant (NKFIH) No. PD  134953 and no. 124950.  
 E-mail: {\tt zoltan.lorant.nagy@ttk.elte.hu}}
 \and Gábor Somlai \thanks{The University of Melbourne, Parkville, VIC, 3010
Australia. G. S. is on an unpaid leave at the Eötvös Loránd University and was supported by OTKA grant no. SNN 152582, Starting 150576.
E-mail: { \tt gabor.somlai@unimelb.edu.au } }}
\date{}

\begin{document}
\maketitle

\begin{abstract} 
A set of points $S \subseteq \mathbb{F}_p^n$ is called \emph{$p$-divisible} if every affine hyperplane in $\mathbb{F}_p^n$ intersects $S$ in  $0 \pmod p$ points.
The Strong Cylinder Conjecture of Ball asserts that  if
 $S$ is a $p$-divisible set of $p^2$ points in $\mathbb{F}_p^3$, then $S$ is a cylinder. In this paper, we show that every $p$-divisible multiset $S$ is both a $\mathbb{F}_p$-linear and $\mathbb{Z}$-linear combination of characteristic functions of cylinders. In addition, the multisets of size $p^2$ are $\Z$-linear combinations of a plane and weighted differences of parallel lines.  
 %provided that the conditions of the strong cylinder conjecture hold. 
    
\end{abstract}



%\section{Introduction}
\section{Introduction}

Let $\AG(n,q)$ denote the $n$-dimensional affine space over the finite field $\mathbb{F}_q$ and  $\PG(n,q)$ the $n$-dimensional projective space over $\mathbb{F}_q$. One can embed the affine space $\AG(n,q)$ into the projective space $\PG(n,q)$. The points of $\PG(n,q)$ that are not in $\AG(n,q)$ form a hyperplane at infinity, denoted by $H_\infty \cong \PG(n-1,q)$. 

For convenience, from now on we identify $\AG(n,q)$ with the vector space $\mathbb{F}_q^n$.  

In the study of finite projective and affine spaces, a fundamental theme is to deduce the structure of a point set from partial information about its combinatorial properties. A central topic in this area is the \emph{direction problem}, which investigates the relationship between the size and the structure of a point set in an affine space and the number of directions determined by the set. Formally, the set of determined directions is defined as follows.

\begin{defi}[Determined directions]
For a set of points $S \subseteq \mathbb{F}_q^n$, the set of \emph{directions $D_S$ determined by $S$} is the set of points at infinity corresponding to the lines passing through at least two points of $S$. The complementary set with respect to $H_{\infty}$ is the set of \emph{non-determined directions}, denoted by $N_S$.
\end{defi}

The direction problem seeks to determine, for a given size $|S|$, the smallest possible number of directions defined by a point set $S$ under certain conditions, and to describe the sets that achieve this minimum. Observe first that if $|S| > q^{n-1}$, then $S$ determines all directions; that is, $D_S = H_\infty$. Hence, the most interesting case arises when $|S| = q^{n-1}$.

With its roots in the work of Rédei \cite{redei}, this problem has led to many deep results and challenging open questions. Most of the known results concern the planar case ($n=2$), see for instance \cite{ball2003, BBBSS, Blokhuis,  Fancsali, Gacs, Gacs2, Ghidelli, negyen, LS, somlai, Sziklai}. In higher dimensions, one of the most notable open problems is the \emph{cylinder conjecture} of Ball \cite{Ball_cilinder}.

\medskip
\noindent
\textbf{The Cylinder Conjecture.}
Simeon Ball proposed the following conjectures, in both a \emph{weak} and a \emph{strong} form, restricted to prime fields $\mathbb{F}_p$. %It was originally formulated in dimension $n=3$.
Let us first define cylinders.
\begin{defi}[Cylinder]
A set $S \subseteq \mathbb{F}_p^3$ is called a \emph{cylinder} if it is the union of $p$ pairwise disjoint parallel lines so it is a set of size $p^2$.
\end{defi}

\begin{conj}[Weak Cylinder Conjecture]
Let $S$ be a set of $p^2$ points in $\mathbb{F}_p^3$, where $p$ is a prime. If the number of non-determined directions of $S$ satisfies $|N_S| \ge p$, then $S$ is a cylinder.
\end{conj}

Observe that the statement of the conjecture is much stronger than the 3-dimensional analogue of the Rédei--Lovász--Schrijver theorem \cite{LS}, proven by Sziklai \cite{Sziklai2}.

\begin{theorem}[Sziklai \cite{Sziklai2}]
Let $S \subseteq \mathbb{F}_p^3$ be a point set of size $|S| = p^2$ for $p > 3$. Then for the number of determined directions $D_S$, one of the following possibilities holds:
\begin{enumerate}
    \item $S$ is a plane and $|D_S| = p + 1$;
    \item $S$ is a cylinder with the (affine part of the) projective triangle as a base, and $|D_S| = \frac{(p + 1)(p+2)}{2}$;
    \item $|D_S| \ge \frac{p^2+5p}{2}$.
\end{enumerate}
\end{theorem}

\medskip
\noindent
The strong form of the conjecture is based on the algebraic property of $p$-divisibility.

\begin{defi}[$p$-divisible set]
A set of points $S \subseteq \mathbb{F}_p^n$ is called \emph{$p$-divisible} if every affine hyperplane in $\mathbb{F}_p^n$ intersects $S$ in a number of points congruent to $0 \pmod p$.
\end{defi}
%Note that the size of a $p$-divisible set is always divisible by $p$, but its cardinality is not necessarily $p^2$.
\begin{conj}[Strong Cylinder Conjecture]
Let $S$ be a set of $p^2$ points in $\mathbb{F}_p^3$, where $p$ is a prime. If $S$ is $p$-divisible, then $S$ is a cylinder.
\end{conj}

The strong form implies the weak one due to the following lemma, which is a special case of a more general theorem of Ball.

\begin{lemma}[\cite{Ball_cilinder}]
Let $S$ be a set of $p^2$ points in $\mathbb{F}_p^3$. If $|N_S| \ge p$, then $S$ is $p$-divisible.
\end{lemma}

The strong cylinder conjecture has been computationally verified for small primes: De Beule, Demeyer, Mattheus, Sziklai \cite{BDMS}, and Kurz and Mattheus \cite{Kurz} established its validity for $p \le 7$.
%\todo[inline]{NZ: szóval ezt direkt írtam így; persze lehet másképp is, de a K-M cikk bejelenti hogy a BDMS cikk hibás, én rossz a biz. Mivel Sam közös szerző, eléggé elhiszem ezt. KG: Mi lenne, ha mindkét munkát említenénk itt? } 
In the former paper, the authors remark that due to Blokhuis, it is known that if a $p$-divisible set of size $p^2$ contains a full line, then it is a cylinder.  On the other hand, the strong conjecture does not hold in general for non-prime fields. A counterexample for $q=8$ was constructed in \cite{Kurz}, consisting of a $8$-divisible set of $64$ points in $\mathbb{F}_8^3$ that is not a cylinder, showing that the prime field restriction is essential. Moreover, even the weak conjecture remains open for $p > 7$.

Recent research has also focused on reformulating and generalizing the conjecture. Kurz and Mattheus \cite{Kurz} interpreted the problem within the framework of divisible linear codes: a linear $[n,k,d]_q$ code $C$ is called \emph{$\Delta$-divisible} if the Hamming weight of every non-zero codeword is a multiple of some integer $\Delta > 1$. The strong cylinder conjecture can then be restated as a question about the structure of certain $p$-divisible codes.

%\medskip
%\noindent

This connection highlights that the geometric structure of point sets with many non-determined directions is closely related to deep additive properties in vector spaces over finite fields.

\medskip
\noindent
\textbf{Our contribution.}
In this paper, we present a relaxation of the strong cylinder conjecture by obtaining the following structural theorems, but as an intermediate step between $ p$-divisible sets and cylinders, we introduce the notation of cylinder type multisets.

\begin{defi}[Multiset] %We call any function $f\colon \mathbb{F}_p^3\to \F_p$ a {\it multiset} (denoted shortly by $S$ (or $S_f$)) and further we use the same expression for functions from $\mathbb{F}_p^3$ to $\Z^{\ge 0}$. The function is called the weight function for the multiset $S$. The size of a multiset $S$ is defined as $|S|=\sum_{x\in \mathbb{F}_p^3} f(x)$. 
We identify subsets and multisets of $\F_p^3$ with their characteristic functions and weight functions, respectively.

A \emph{(integer-valued) multiset} $S$ on $\F_p^3$ is a function
$f_S:\F_p^3\to \Z_{\ge0}$.
Its size is defined by
\[
|S|=\sum_{x\in\F_p^3} f_S(x).
\]

More generally, we allow $\F_p$-valued weight functions
$f_S:\F_p^3\to \F_p$.
In this case, all sums and equalities are understood modulo $p$,
and $|S|:=\sum_x f_S(x)\in\F_p$ denotes the total weight of $S$.

%\red{ez azert igy nem jo, mert keveredik az Fp, es a Z, az elsoben meg nincs rendezes}
\end{defi}
\begin{defi}[$p$-divisible multiset]
We say that a multiset of points $S \subseteq \mathbb{F}_p^3$ is called \emph{$p$-divisible} if every affine hyperplane in $\mathbb{F}_p^n$ intersects $S$ in a number of points congruent to $0 \pmod p$.
\end{defi}
Note that the size of a $p$-divisible multiset is always divisible by $p$, but its cardinality is not necessarily $p^2$.

\begin{defi}[Cylinder type multisets]
We say that a multiset $S \subset \F_p^3$ is of \emph{cylinder type} if its weight function 
$f_S$ can be written as a linear combination of characteristic functions of $p$ affine (might be overlapping) lines.
\end{defi}
For a pair of lines $\ell_1$ and $\ell_2$ we write $\ell_1 \parallel \ell_2$ if $\ell_1$ and $\ell_2$ are parallel. 
%Equivalently, the function $\mathbb{1}_H$ is a so-called \emph{cylindric function}, i.e.,
%it depends linearly on one variable and is constant on the lines parallel to a fixed direction. 

%\begin{theorem}%[Main result]
%\label{thm:cilinderZ_plinkomb}
%Let $S\subseteq \mathbb{F}_p^3$ be a multiset such that $|H\cap S|\equiv 0 \pmod p$ for every affine plane $H$. Then $S$ is an $\mathbb{F}_p$-linear combination of cylindrical (multi)sets consisting of $p$ parallel lines.
%\end{theorem}
\begin{theorem}\label{thm:cilinderZ_plinkomb}
Let $S\subseteq \mathbb{F}_p^3$ be a $p$-divisible multiset (over $\F_p$). 
%such that $|H\cap S|\equiv 0 \pmod p$ for every affine plane $H$. 
Then the following hold. The set $S$ lies in the $\mathbb{F}_p$–linear span of  any of the following:
\begin{enumerate}
\item cylinders, 
\item cylinder type multisets,
\item differences of parallel lines, i.e.,
%that is, unions (possibly with multiplicities) of $p$ parallel lines. 
%Equivalently,
\[
S \in \text{Span}_{\F_p}\{\1_{\ell_1}-\1_{\ell_2} ~ \colon ~  \ell_1\parallel \ell_2\}.
\]
\end{enumerate}
\end{theorem}

Now we consider $\Z$-linear combinations of cylinders to strengthen this result as follows.
\begin{theorem}\label{thm:cilinderZlinkomb}
Let $S\subseteq \mathbb{F}_p^3$ be a $p$-divisible set of size $p^2$. 
% with $|H\cap S|\equiv 0 \pmod p$ for every affine plane $H$.
Then $S$ is a $\mathbb{Z}$-linear combination of cylinders. Alternatively, $S$ is the sum of the characteristic function of a plane and an element of  $\text{Span}_{\Z}\{\1_{\ell_1}-\1_{\ell_2}\mid \ell_1\parallel \ell_2\}$.
\end{theorem}

One may try to extend Theorem~\ref{thm:cilinderZlinkomb} to
$p$-divisible multisets without any restriction on their size.
However, this is not possible in general.

Indeed, a single point of $\F_p^3$ with weight $p$ is a $p$-divisible
multiset, but its total weight equals $p$ and hence is not divisible
by $p^2$.
On the other hand, every cylinder has a total weight divisible by $p^2$,
and therefore any $\Z$-linear combination of cylinders also has this
property.
Consequently, the conclusion of Theorem~\ref{thm:cilinderZlinkomb}
    fails without the size assumption. However we prove the following. 
    \begin{theorem}[Multiset version of the cylinder theorem]
\label{thm:multiset-cylinder}\label{thm:multisetcylinder}
Let $w:\F_p^3\to\Z_{\ge 0}$ be a $p$--divisible multiset, that is, for every affine plane
$H\subseteq\F_p^3$,
\[
\sum_{x\in H} w(x)\equiv 0 \pmod p.
\]
Assume moreover that
\[
\sum_{x\in\F_p^3} w(x)=p^2.
\]
Then for any affine plane $H\subseteq\F_p^3$ it holds that
\[
w-\1_H \in \Span_{\Z}\{\1_{\ell_1}-\1_{\ell_2}\mid \ell_1\parallel \ell_2\}.
\]
Equivalently, $w$ is a $\Z$--linear combination of cylinders.
\end{theorem}

Our paper is organized as follows. 
Section \ref{sec:2} introduces basic notation and recalls well-known observations used throughout the paper. Sections \ref{sec:3} and \ref{cylinders} establish results over $\mathbb{F}_p$, interpreting $p$-divisible sets via characteristic polynomials and completing the proof of Theorem~11. Section \ref{sec:5} extends these results from $\mathbb{F}_p$ to $\mathbb{Z}$. Section \ref{sec:6} reveals an interesting connection with vanishing sums of roots of unity.  Section \ref{sec:7} explains how the cylindrical set conjecture is related to R\'edei's conjecture.

\section{Notation and technique}\label{sec:2}

Throughout the paper, $p$ denotes a prime number and $\F_p^n$
the $n$-dimensional affine space over the finite field $\F_p$.
Affine one- and two-dimensional subspaces of $\F_p^n$ will be called
lines and planes, respectively.

As introduced above, we identify sets and multisets in $\F_p^n$
with their characteristic functions and weight functions.
In particular, a (usual) set corresponds to a $\{0,1\}$-valued function,
which may be regarded either as a function into $\Z_{\ge 0}$ or into $\F_p$,
depending on the context.

The key observation underlying our proofs is that the condition that
a multiset $S$ intersects every affine plane $H$ in a multiple of $p$
can be reformulated as an orthogonality condition modulo $p$
between the corresponding weight functions.
This allows us to treat functions from $\F_p^n$ to $\F_p$
as $n$-variable polynomials over $\F_p$
and to apply elementary algebraic tools.

We begin with the following well-known identity.

%We treat functions from $\F_p^n$ to $\F_p$ as $n$-variable polynomials. 
%Let us start with the following trivial and well-known observation. 

\begin{prop}\label{powersum} Suppose that $s$ is a nonnegative integer. Then
  $$\sum_{a\in \F_p}a^s=
  \begin{cases}
      -1 & \text{if } (p-1) \mid s \mbox{ and } s>0 \\
    \   0 & \text{else. } 
  \end{cases}$$
\end{prop}

Observe that Proposition \ref{powersum} implies the following more general result. 
\begin{prop}\label{lem:integral}
Let $m(x_1,x_2, \ldots x_n)=x_1^{k_1}x_2^{k_2} \ldots x_n^{k_n}$ be a monomial in $ \F_p[x_1,x_2, \ldots, x_n]$. Then 
$$\sum_{(a_1,a_2, \ldots, a_n)  \in \F_p^k} m(a_1,a_2, \ldots , a_n)=  \sum_{a_i \in \F_p} \prod_{i=1}^n a_i^{k_i}=  \begin{cases}
    \  (-1)^n & \text{if }  p-1 \mid k_i>0 \mbox{  for every } 1 \le i \le n ,  \\
     \  0 & \text{else. } 
  \end{cases}$$ 
%$$\sum_{a_1,a_2, \ldots, a_n \in \F_p} m(a_1,a_2, \ldots, a_n) \ne 0 $$ if and only if $p-1 \mid k_i$ for $i=1,2,\ldots, n$.
\end{prop}

%\begin{prop}\label{prop:monomialsum}
%Let $m(a_1,a_2, \ldots , a_k):=a_1^{\alpha_1} a_2^{\alpha_2} \ldots a_k^{\alpha_k}$ be a monomial. Then  
%$$\sum_{(a_1,a_2, \ldots, a_k)  \in \F_p^k} m(a_1,a_2, \ldots , a_k)=  \sum_{a_i \in \F_p} \prod_{i=1}^k a_i^{\alpha_i}=  \begin{cases}
%    \  (-1)^k & \text{if }  \alpha_i = p-1 \mbox{  for every } 1 \le i \le k ,  \\
%     \  0 & \text{else. } 
%  \end{cases}$$ 
%\end{prop}

\begin{defi}[Orthogonal functions]
%We define a bilinear form on the $\F_p$ vector space of functions from $\F_p^n \to \F_p$.
Let $f,g\colon \F_p^n \to \F_p$. We define a bilinear form
\[
    \langle f,g\rangle := \sum_{x\in \F_p^n} f(x)g(x).
\]
We say that $f$ and $g$ are \emph{orthogonal}, and write $f\perp g$, if and only if $\langle f,g\rangle = 0$.
\end{defi}

In particular, for functions $f,g\colon \F_p^3 \to \F_p$, the orthogonality is understood in the same sense:
\[
    f\perp g \quad \Longleftrightarrow \quad \sum_{x\in \F_p^3} f(x)g(x)=0.
\]
For a family of affine planes $\mathcal{H}$ in $\F_p^3$, we may write
\[
    f \in \operatorname{Span}\{\, \mathbb{1}_H \mid H \text{ affine plane}\}^{\perp}
    \quad \text{if and only if} \quad
    f\perp \mathbb{1}_H \text{ for all } H\in \mathcal{H}.
\]

\section{Span of planes}\label{sec:3}

In this section we work with $\F_p$--valued weight functions
$f\colon \F_p^3\to\F_p$,
which arise naturally when the $p$--divisibility condition is interpreted
modulo~$p$.
(In particular, characteristic functions of sets are special cases.)

Let $\mathcal H$ denote the family of affine planes in $\F_p^3$, and for
$H\in\mathcal H$ let $\1_H\colon\F_p^3\to\{0,1\}\subseteq\F_p$ be its
characteristic function.
We consider the $\F_p$--linear subspace
\[
\mathcal S_1 := \Span_{\F_p}\{\1_H : H\in\mathcal H\}
\subseteq \{\, f\colon\F_p^3\to\F_p \,\}.
\]
A function $f$ satisfies $f\perp \1_H$ for every affine plane $H$
if and only if $f\in \mathcal S_1^{\perp}$.

The description of $\mathcal S_1$ obtained below will allow us, in the next
section, to determine its orthogonal complement $\mathcal S_1^{\perp}$
explicitly.
This characterization is the key algebraic ingredient in the proof of
Theorem~\ref{thm:cilinderZ_plinkomb}.

We identify the functions $\F_p^3\to\F_p$ with the polynomials
$\F_p[x,y,z]/_(x^p-x,y^p-y,z^p-z)$, or, looking at it another way, with polynomials $\F_p[x,y,z]$ whose degree in  each variable is at most $p-1$.

\begin{prop}\label{prop:gene_lines}
Let us identify 3-variable polynomials over $\F_p$ by their polynomial function from $\F_p^3$ to $\F_p$. Then 
\[
   \mathcal{S}_1 = \{\, m(x,y,z) \in \F_p[x,y,z] : \deg(m)\le p-1 \,\}.
\]
\end{prop}

\begin{proof}
First observe that for every $a,b,c,d\in\F_p$, the function
\[
   \1_H := 1-(ax+by+cz+d)^{p-1}
\]
is the characteristic function of the affine plane $H$, whose points are defined by the equation $ax+by+cz+d=0$.
%\todo{Ez be van valahol vezetve? Miert H, en elsore V-re gondolnek (variety), de a jeloles mindegy. SG }. 
Indeed,  $t^{p-1}=1$ for $t\ne 0$ holds in $\F_p$, hence $\1_H(x,y,z)=1$ exactly on the plane $H$ and $0$ elsewhere. Therefore,
\[
   \mathcal{S}_1 = \mathrm{Span}_{\F_p}\{\, 1-(ax+by+cz+d)^{p-1} : a,b,c,d\in\F_p \,\},
\]
and every element of $\mathcal{S}_1$ is a polynomial of degree at most $p-1$.

To show the converse inclusion, it suffices to prove that every monomial
$x^{\alpha}y^{\beta}z^{\gamma}$ with $\alpha+\beta+\gamma\le p-1$
can be obtained as a linear combination of the characteristic functions of affine planes.

Fix integers $0\le s,t,u \le p-1$ and consider the weighted sum of characteristic functions of planes. The weights are defined by a polynomial function arising from a monomial %$v(a,b,c)=
$a^sb^tc^u$. Then we get
%\todo[inline]{KG:Ez csak a $d=0$-t kezeli és kell még valami a $s=t=u=0$-ról is.\\ NZ: beírtam most $d$-t.}
\begin{equation}\label{eq:sum_planes}
\begin{split}
Q_{s,t,u}(x,y,z)
   &= \sum_{a,b,c\in\F_p} \!\bigl(1-(ax+by+cz+d)^{p-1}\bigr)a^sb^tc^u\\
   &= \sum_{a,b,c\in\F_p}a^sb^tc^u
     -\sum_{a,b,c\in\F_p}\sum_{\substack{i,j,k\ge0\\ i+j+k\le p-1}}
       \binom{p-1}{i,j,k}a^{i+s}b^{j+t}c^{k+u}d^{(p-1-i-j-k)}x^iy^jz^k,
\end{split}
\end{equation}
where $\binom{p-1}{i,j,k}=\dfrac{(p-1)!}{i!\,j!\,k!\,(p-1-i-j-k)!}$ denotes the multinomial coefficient, which is not divisible by $p$, and $d\in \F_p$ is a fixed element.

By Proposition~\ref{lem:integral} (the standard identity showing $\sum_{x\in\F_p}x^r=0$  holds unless $p-1\mid~ r$), the first term on the second line of \eqref{eq:sum_planes} vanishes except when $p-1$ divides $s,t$ and $u$. Moreover, this term is constant (on $\F_p^3$), which is clearly a linear combination of characteristic functions of planes. 
Interchanging the order of summation in the second term, we obtain
\[
% Q_{s,t,u}(x,y,z)
  % = 
  % - 
   \!\!\sum_{\substack{i,j,k\ge0\\ i+j+k\le p-1}}\!\!
       \binom{p-1}{i,j,k}x^iy^jz^kd^{(p-1-i-j-k)}
       \!\!\sum_{a,b,c\in\F_p}\! a^{i+s}b^{j+t}c^{k+u}.
    \]
Applying Proposition~\ref{lem:integral} again, we get that  the inner sum is nonzero precisely when
\[
   p-1 \mid i+s,\quad p-1 \mid j+t,\quad p-1 \mid k+u.
\]
For given $(s,t,u)$ with $0\le s,t,u\le p-1$, there is exactly one triple $(i,j,k)$ with $i+j+k\le p-1$, 
%\todo{En most ezt hirtelen nem latom, miert is fontos ez a feltetel. nyilvan az. SG}
 %there is exactly one triple $(s,t,u)$ with $0\le s,t,u<p-1$ 
satisfying
\[
   i+s=j+t=k+u=p-1, 
\]
except if $s=t=u=p-1$. In this exceptional case, $x^iy^jz^k$ is constant, which again is the linear combination of characteristic functions of planes.

%\todo{Szerintem itt van nemi tennivalo, mert egyreszt nem p-1 osszeget akarunk mindenhol, hanem bizonyos helyekre 0-at. Viszont ha mind 0, akkor van a konstans, ahol az osszeg talan inkabb 0, SG}
Choosing these exponents $s=p-1-i,t=p-1-j,u=p-1-k$, all inner sums vanish except one, and $Q_{s,t,u}$ reduces (up to adding and multiplying by nonzero scalar) to the monomial $x^iy^jz^k$.

Thus, for every monomial $x^{i}y^{j}z^{k}$ of total degree at most $p-1$, there exists a linear combination of characteristic functions of affine planes that equals it. Therefore, all polynomials of degree $\le p-1$ belong to $\mathcal{S}_1$.

Since we have proven the reverse inclusion $\mathcal{S}_1\subseteq\{m \in \F_p[x,y,z] :\deg(m)\le p-1\}$ also holds, the proof is complete.
\end{proof}

 \section{Cylinder type sets over \texorpdfstring{$\F_p$}{}}\label{cylinders}

As it was noticed earlier, the linear function space we would like to understand is the one that is orthogonal to $\mathcal{S}_1$. 
Before describing it precisely, let us recall the geometric meaning of these spaces.

%We say that a set $H \subset \F_p^3$ is of \emph{cylinder type} if its characteristic function $\mathbb{1}_H$ can be written as a linear combination of characteristic functions of affine lines. Equivalently, the function $\mathbb{1}_H$ is a so-called \emph{cylindric function}, i.e., it depends linearly on one variable and is constant on the lines parallel to a fixed direction.

As a consequence of the previous proposition (Proposition \ref{prop:gene_lines}), the characteristic function of a cylinder type set $H$ 
is orthogonal to $\mathcal{S}_1$, and hence to every monomial of degree at most $p-1$. 
It is also easy to verify, using Proposition~\ref{lem:integral}, that all monomials of degree at most $2p-3$ satisfy this orthogonality condition. 
Thus these monomials are contained in $\mathcal{S}_1^{\perp}$.
We now show that, in fact, $\mathcal{S}_1^\perp$ is spanned by these monomials.

\begin{prop}\label{prop:S_1meroleges}
$$
\mathcal{S}_1^{\perp}=\left\{ m(x,y,z) \in \F_p[x,y,z] \mid \deg(m) \le 2p-3 \right\}.
$$
\end{prop}

\begin{proof}
Let us write the characteristic polynomial of an element of $\mathcal{S}_1^\perp$ 
as a polynomial $m$ which is a sum of monomials of the form $x^iy^jz^k$, where $0 \le i,j,k \le p-1$. 

Assume that the characteristic function of $H$ contains a monomial $x^iy^jz^k$ of degree at least $2p-2$. Further, we may also assume that the total degree $i+j+k$ of this monomial is maximal.  
It follows from Proposition~\ref{lem:integral} that such a monomial is not orthogonal to $x^{p-1-i}y^{p-1-j}z^{p-1-k}$, 
while $x^{p-1-i}y^{p-1-j}z^{p-1-k}$ is orthogonal to every other monomial appearing in $m$ by Proposition \ref{lem:integral}, and by our maximality assumption on the degree of the monomial $x^iy^jz^k$. 
Therefore, no monomial of degree at least $2p-2$ can appear with a nonzero coefficient in $m$.

On the other hand, it is clear that every polynomial of degree at most $2p-3$ 
is orthogonal to $\mathcal{S}_1$, hence is contained in $\mathcal{S}_1^\perp$. 
This completes the proof.
\end{proof}

Now we prove a proposition that is not necessarily needed in order to finish the proof of our theorems. This result shows that a suitable generalisation of Proposition \ref{prop:gene_lines} to higher dimensional spaces is also possible to 
establish.
\begin{prop}\label{prop:egyeneseklinkomb}
 $$\mathcal{S}_2=\left\{ \mathbb{1}_\ell \mid \ell \in \mathcal{L}     \right\}=\left\{ m(x,y,z) \in \F_p[x,y,z] \mid \deg(m) \le 2p-2 \right\}.$$   
\end{prop}
\begin{proof}
    We  write the characteristic functions of lines as the product of characteristic functions of two planes. There is no canonical way of doing so but we may try to find a way that produces a lot of monomials of degree at most $2p-2$ but there is an easier approach to follow.  

    Observe that the intersection of two planes is always a linear combination of lines. Indeed, if the two planes are not parallel, then the intersection is a line. If the two planes are parallel and disjoint, when the intersection is the emptyset, then it is a trivial linear combination of lines. Finally, if the two planes coincide, then the intersection is a plane which is the sum of $p$ parallel lines.  

    Thus if we write $$\left( \sum_{\pi \in \mathcal{P}_1} \pi \right)\left( \sum_{\tilde{\pi} \in \mathcal{P}_2} \tilde{\pi} \right),$$
    where $\mathcal{P}_1$ and $\mathcal{P}_2$ are multisets of planes, and $\pi$ and $\tilde{\pi}$ denote both a plane and its characteristic function, then the product will be a linear combination of lines. 

Every monomial of degree at most $2p-2$ can be written as the product of two polynomials of degree at most $p-1$. Thus, by Proposition \ref{prop:gene_lines} we have that $S_2$ contains every monomial of degree at most $2p-2$. Since lines coincide with polynomials of degree $2p-2$ we obtain the required equality.  
\end{proof}
A standard way to proceed is by investigating the 'leading coefficient' of polynomials of degree at most $2p-2$, or in other words, a linear combination of lines.  This could also lead to a proof of Theorem \ref{thm:cilinderZ_plinkomb}.% whose proof is simpler than we initially thought. \todo{szerintem ez sidenote, nem kell a cikkbe, t.i. hogy mi hogy mikor jöttünk rá vmire}

\begin{prop}\label{prop:egyenes kulonbseg}
The $\F_p$--linear span of cylinder type multisets coincides with the $\F_p$--linear span of differences of parallel lines:
\[
\mathrm{Span}_{\F_p}\{\text{cylinder type multisets}\}
=\mathrm{Span}_{\F_p}\{\1_{\ell_1}-\1_{\ell_2}\mid \ell_1,\ell_2 \text{ are parallel lines}\}.
\]
\end{prop}

\begin{proof}
We first show that every difference of two parallel lines can be written as a difference of two cylinder type multiset.  
Indeed, consider two cylinder type multisets $C_1$ and $C_2$ whose supports consist of $p-1$ common parallel lines and differ only on two additional parallel lines $\ell_1$ and $\ell_2$.  
Then their weight functions differ exactly on $\ell_1$ and $\ell_2$, hence
\[
f_{C_1}-f_{C_2}=\1_{\ell_1}-\1_{\ell_2},
\]
so each $\1_{\ell_1}-\1_{\ell_2}$ is contained in the span of cylinders.

Conversely, let $C$ be any cylinder type multiset with fixed direction $v$. Then its weight function $f_C$ can be written as
\[
f_C=\sum_{\ell\parallel v} c(\ell)\,\1_\ell,
\]
where the sum runs over all $v$–parallel lines denoted by $\ell\parallel v$ and $c(\ell)\in\F_p$.  
Fixing one line $\ell_0\parallel v$, we can rewrite this as
\[
f_C=\sum_{\ell\parallel v} c(\ell)(\1_{\ell}-\1_{\ell_0})+\Bigl(\sum_{\ell\parallel v} c(\ell)\Bigr)\1_{\ell_0}=\sum_{\ell\parallel v} c(\ell)(\1_{\ell}-\1_{\ell_0})
\]
since $\sum_{\ell \parallel v} c(\ell)=0$ in $\F_p$.
%The first term is a linear combination of differences of parallel lines, while the second one is a multiple of $\1_{\ell_0}$, which itself lies in their span since
%\[
%\sum_{\ell\parallel v}(\1_\ell-\1_{\ell_0})=\1-\1_{\ell_0}.
%\]
%Hence $\1_C$ also belongs to the span of such differences. 
%This proves the equality of the two spans.
\end{proof}

%The following proposition .

%\todo[inline]{lin comb cilinder, cilinder functions, diff. parallel lines- kéne egy megyjegyzés}

The following proposition, which is the last ingredient to prove Theorem \ref{thm:cilinderZ_plinkomb}, shows that the algebraic description of
$\mathcal S_1^\perp$ obtained above admits an equivalent geometric formulation
in terms of cylinders and differences of parallel lines.

\begin{prop}\label{prop:vege}
    $$S_1^{\perp}=Span_{\F_p}\left\{ \mbox{cylinder type multisets} \right\}=Span_{\F_p} \left\{\1_{\ell_1}-\1_{\ell_2} \mid \ell_1 \mbox{ and } \ell_2 \mbox{ are parallel lines} \right\}.$$
\end{prop}

\begin{proof}
We may represent the difference of two parallel lines as the intersection of a plane with the difference of two parallel planes. Thus, any differences of parallel lines can be written in the following general form 
\begin{equation*}
%\begin{split}
    \left( 1- \left( a_1x+b_1 y+c_1 z +d \right)^{p-1} - 1+ \left( a_1x+b_1 y+c_1 z +e \right)^{p-1}   \right)    \left( 1-\left( a_2x +b_2y+c_2z+g \right)^{p-1}\right).
 %   \end{split}
\end{equation*}
It is easy to see that such a polynomial is of degree at most $2p-3$ which shows that the subspace generated by the differences of pairwise orthogonal lines is contained in the set of polynomials of degree at most $2p-3$, which is equal to  $Span_{\F_p}\left\{ \mbox{cylinder type multisets} \right\}$ by Proposition \ref{prop:S_1meroleges} and Proposition \ref{prop:gene_lines}.

It remains to see that every monomial of degree at most $2p-3$ can be written as a linear combination of differences of parallel lines. In order to establish this statement, we calculate the following formal expression.
     \begin{equation}\label{eq:sikkulonbsegszorsik}
\begin{split}
&\sum_{a_1,b_1,c_1 \in \F_p}a^{s_1}b^{t_1}c^{u_1} \left( \left(1- \left( a_1x+b_1 y+c_1 z +1 \right)^{p-1}   \right)-\left(1- \left( a_1x+b_1 y+c_1 z \right)^{p-1}   \right) \right) \cdot \\
 & \sum_{a_2,b_2,c_2\in \F_p} a_2^{s_2}b_2^{t_2}c_2^{u_2} \left(1-(a_2 x+b_2 y+c_2 z)^{p-1} \right) .
\end{split}
\end{equation}
Now it follows from Proposition \ref{prop:gene_lines} that any monomial of degree at most $p-1$ can be written as the linear combination polynomials corresponding to planes. A similar computation shows that monomials of degree at most $p-2$ can be written as the linear combination of differences of parallel planes just as we do in \eqref{eq:sikkulonbsegszorsik}. Using the distributive law we may see that the formula in equation \eqref{eq:sikkulonbsegszorsik} is a linear combination of differences of parallel lines, and planes. Now, any plane is the $\F_p$-linear combination of differences of parallel lines, finishing the proof of Proposition \ref{prop:vege}. 
\end{proof}

\textbf{Proof of Theorem \ref{thm:cilinderZ_plinkomb}}
Let $f_S$ be the weight functions (over $\F_p$) of a $p$-divisible multiset. The hypothesis that $|H \cap S|\equiv 0 \pmod p$ for every affine plane $H$
is equivalent to
\[
\sum_{x \in \F_p^3} f_S(x)\,\1_H(x) = 0 \quad \text{in } \F_p
\quad\text{for all affine planes }H,
\]
i.e.
\[
f_S \perp \1_H \quad \text{for all planes } H,
\] 
where orthogonality is taken over $\F_p$. % meaning
%\(\langle f,g\rangle := \sum_{x\in\F_p^3} f(x)g(x)=0 \text{ in } \F_p.\)
In other words,
\[
f_S \in \mathcal{S}_1^\perp,
\qquad
\mathcal{S}_1 := \text{Span}_{\F_p}\{\1_H : H \text{ affine plane in } \F_p^3\}.
\]

%By Proposition~\ref{prop:gene_lines}, $\mathcal{S}_1$ is exactly the space of all polynomials
%in $\F_p[x,y,z]$ of degree at most $p-1$. Hence $\mathcal{S}_1^\perp$ is the space of all
%functions orthogonal to every polynomial of degree $\le p-1$.
Proposition~\ref{prop:S_1meroleges} identifies this orthogonal complement:
\[
\mathcal{S}_1^\perp
=
\{ m(x,y,z)\in\F_p[x,y,z] : \deg m \le 2p-3 \}.
\]

We now describe $\mathcal{S}_1^\perp$ geometrically. 
Proposition~\ref{prop:egyenes kulonbseg} shows that the $\F_p$–linear span of cylinder type multisets
(that is, characteristic functions of unions of $p$ parallel lines, possibly with multiplicity)
coincides with the $\F_p$–linear span of all differences $\1_{l_1}-\1_{l_2}$ of parallel lines.
Finally, Proposition~\ref{prop:vege} states that
\[
\mathcal{S}_1^\perp
=
\text{Span}_{\F_p}\{\text{cylinder type multisets}\}
=
\text{Span}_{\F_p}\{\1_{\ell_1}-\1_{\ell_2} : \ell_1 \parallel \ell_2\}.
\]

Since $f_S \in \mathcal{S}_1^\perp$, this implies that $f_S$ is an $\F_p$–linear combination
of cylinder sets; equivalently, it is an $\F_p$–linear combination of differences
of parallel lines. This proves the theorem.
\qed

\section{Cylinder type sets over \texorpdfstring{$\Z$}{}}\label{sec:5}
%Integer liftings of cylinder relations}

In the previous section, we described the orthogonal complement
$\mathcal S_1^\perp$ as the $\F_p$--linear span of cylinder type
configurations, or equivalently, as the $\F_p$--linear span of
differences of parallel affine lines.

In the present section, we strengthen this description by passing from
$\F_p$--linear relations to integer ones.
More precisely, we show that $\F_p$--linear combinations of characteristic functions of parallel lines can be lifted to integer combinations.
In the case of sets, this lifting can be carried out in a constructive
way, leading to $\{0,1\}$--valued characteristic functions.
For multisets, an analogous---and formally stronger---statement holds,
although our corresponding argument is no longer constructive.
\subsection{Auxiliary lemmata}
\begin{lemma}[Point difference lemma]\label{lemma:point-diff}
Let $a,b\in\F_p^3$ be distinct points.
Then
\[
e_{a,b}:=p\,\1_a-p\,\1_b\in
\Span_{\Z}\{\1_{\ell_1}-\1_{\ell_2}~\colon ~ \ell_1\parallel \ell_2\}.
\]
\end{lemma}

\begin{proof}
Fix an affine plane $H$ containing $a$ and $b$.
For every line $\ell\subset H$ through $a$ not containing $b$,
there exists a unique line $\ell'\subset H$ through $b$ parallel to $\ell$.
Assign weight $+1$ to $\ell$ and $-1$ to $\ell'$.
Summing over all such pairs yields a function supported on $\{a,b\}$,
taking value $p$ at $a$ and $-p$ at $b$.
\end{proof}

\begin{lemma}[$p$--times zero--sum lemma]\label{lemma:p-zero-sum}
Let $u:\F_p^3\to\Z$ satisfy $\sum_{x\in\F_p^3}u(x)=0$.
Then
\[
p\,u\in
\Span_{\Z}\{\1_{\ell_1}-\1_{\ell_2}~ \colon ~ \ell_1\parallel \ell_2\}.
\]
\end{lemma}

\begin{proof}
Since $\sum_{x\in \F_p^3} u(x)=0$, the function $u$ can be written as a finite sum
\[
u=\sum_{k=1}^N(\1_{a_k}-\1_{b_k})
\]
of point differences ($a_k,b_k \in \F_p^3$).
Multiplying by $p$ and applying Lemma~\ref{lemma:point-diff} to each term
proves the claim.
\end{proof}

\subsection{Proof of Theorem~\ref{thm:cilinderZlinkomb}}

Let $S\subseteq\F_p^3$ be a $p$--divisible \emph{set} of size $p^2$ and
$f=\1_S$.

\medskip
\textbf{Step 1: Reduction modulo $p$ and lifting.}
Since $f\in\mathcal S_1^\perp$, by the results of the previous section we may write
\[
f\equiv \1_H+\sum_i \alpha_i(\1_{\ell_i}-\1_{\ell_i'})
\pmod p,
\]
where $H$ is an affine plane and $\ell_i\parallel \ell_i'$ are parallel lines.
Lifting the coefficients $\alpha_i\in\F_p$ to integers in
$\{0,1,\dots,p-1\}$ yields an integer--valued function $g$ with
$g\equiv f\pmod p$.

\medskip
\textbf{Step 2: Eliminating negative values.}
The function $g$ satisfies
\begin{equation}\label{eq:lemma5.2}
\sum_{x \in \F_p^3} g(x)=p^2, \mbox{ and}
\qquad
g(x)\equiv 0\textrm{ or }1\pmod p.
\end{equation}
If $g$ is nonnegative, then necessarily $g=\1_S$ and we are done.
Otherwise, choose $c$ with $g(c)<0$ and $d$ with $g(d)\ge p$.
Replacing $g$ by $g+e_{c,d}$, where $e_{c,d}$ is defined in
Lemma~\ref{lemma:point-diff}, preserves both congruence and total sum conditions in equation $\eqref{eq:lemma5.2}$,
while strictly decreasing
\[
\sum_{x~:~g(x)<0}|g(x)|.
\]
Iterating this procedure yields a $\{0,1\}$--valued function.

\medskip
\textbf{Step 3: Conclusion.}
\medskip
\textbf{Step 3: Conclusion.}
The resulting function can be written as the characteristic function of
an affine plane plus an integer linear combination of differences of
parallel lines.
Since an affine plane is the disjoint union of $p$ parallel lines, its
characteristic function is that of a cylinder.

Moreover, the differences of parallel lines lie in the $\Z$--linear span of
cylinders, as follows from the fact that such differences can be realized
as differences of two cylinders with a common $(p-1)$--line intersection.
Consequently, $\1_S$ belongs to the $\Z$--linear span of cylinders
(and hence also of cylinder--type multisets).
\qed

\qed

\subsection{Proof of Theorem~\ref{thm:multiset-cylinder}}

Let $w:\F_p^3\to\Z_{\ge0}$ be a $p$--divisible multiset with total weight
$p^2$.

\medskip
\textbf{Step 1: Reduction modulo $p$ and lifting.}
Let $r$ be a function from $\F_p^3$ to $\F_p$ defined as  $r(x)\equiv w(x)\pmod p$ for every $x\in \F_p^3$.
Then $r\in\mathcal S_1^\perp$, hence
\[
r\in
\Span_{\F_p}\{\1_{\ell_1}-\1_{\ell_2}~ \colon ~\ell_1\parallel \ell_2\}.
\]
Choosing such a representation and lifting coefficients to
$\{0,1,\dots,p-1\}$ yields
\[
g\in
\Span_{\Z}\{\1_{\ell_1}-\1_{\ell_2} ~ \colon ~ \ell_1\parallel \ell_2\},
\qquad
g\equiv r\pmod p,
\]
with $\sum_{x\in \F_p^3} g(x)=0$.

\medskip
\textbf{Step 2: Extracting the $p$--multiple part.}
Since $w-g(x)$ is divisible by $p$ for every $x \in \F_p^3$, there exists $t:\F_p^3\to\Z$ such that
\[
w-g=p\,t.
\]
Simple calculation shows $\sum_{x\in \F_p^3} t(x)=p$.
Fix an affine plane $H$ and a line $\ell_0\subset H$.
Then $u:=t-\1_{\ell_0}$ satisfies $\sum_{x\in \F_p^3} u(x)=0$, hence by
Lemma~\ref{lemma:p-zero-sum},
\[
p\,u\in
\Span_{\Z}\{\1_{\ell_1}-\1_{\ell_2} ~ \colon ~ \ell_1\parallel \ell_2\}.
\]

\medskip
\textbf{Step 3: Removing the plane component.}
Writing $H$ as the disjoint union of the $p$ lines parallel to $\ell_0$ gives
\[
\1_H-p\,\1_{\ell_0}\in
\Span_{\Z}\{\1_{\ell_1}-\1_{\ell_2} ~ \colon ~ \ell_1\parallel \ell_2\}.
\]
Combining the above relations yields
\[
w-\1_H =g+pu-(\1_H-p\, \1_{\ell_0} )\in
\Span_{\Z}\{\1_{\ell_1}-\1_{\ell_2} ~ \colon ~ \ell_1\parallel \ell_2\},
\]
so $w$ is a $\Z$--linear combination of cylinders.
\qed

%\section{Dimension}

%\todo[inline]{Ádám-féle dimenziószámítás}

%In the next two theorems, we calculate the dimension of the subspace generated by the characteristic function of cylinders over $\F_p$ and  describe  min generating system  over $\mathbb{Z}$. 

%\begin{theorem}
%   For the $S \in \mathcal{S}$ cylinder sets it holds that $ \mod p$
%   \begin{displaymath}
%       a
%   \end{displaymath}
%\end{theorem}

\section{Connection to vanishing sums of roots of unities}\label{sec:6}
It is natural to examine finite sums of the complex roots of the unity and ask ourselves whether the result is zero. Classical treatment of the problem is due to de Bruin \cite{deb}, Rédei \cite{redeivanishing}, Schoenberg \cite{schoen},  and Lam-Leung \cite{Lam}, whose results we describe below. 

Further, this question is interesting since the values of characters of the complex irreducible representations of finite groups are the sums of roots of unity. The Fourier transform of a function on a finite abelian group is described using characters, and this is a fundamental and classical tool in understanding tilings of finite abelian groups \cite{F1,KMMS, LL1, LL2, LM, Mal}. 
%A similar phenomenon arises in the case of nonabelian and rather symmetric groups as well, but these are recent results. 
Another recent result in this topic is due to %Christopher
Herbig \cite{herbig}, which gives an upper bound on the number of minimal $k$-term vanishing sums of roots of unity containing $1$ as a summand and uses it to disprove a conjecture of Hung and Tiep \cite{tiep}. Notice that the finiteness of this number is already non-trivial. 

Since the sums we consider are finite \footnote{infinite would not converge anyway}, we may restrict ourselves to $m$'th roots of unity for large enough positive integer $m$.
It is easy to see that the sum of all $k$'th roots of unity is zero if $k \ge 2$, and we apply this observation for the prime  divisors $p_1,p_2, \ldots, p_k$ of $m$. 

We can look at the $m$'th roots of unity as a multiplicative group which is a cyclic group of order $m$. The $k$'th roots of unity form a subgroup for every $k \mid m$ and what we have seen is that sums of elements of any subgroup is equal to 0. Then the same holds if we add elements of any of these cosets of nontrivial subgroups. 

We treat the cyclic group $\Z_m$ as the direct sum of cyclic groups of order $p_i^{\alpha_i}$ and then we may embed $\Z_m$ into $\mathbb{R}^k$ as lattice points. Notice that subgroups of order $p_i$ coincide with chosen grid points in a line parallel to one of the axis. The same actually holds for cosets of subgroup or order $p_i$ and their sum is also equal to zero. These subsets of the grid points will be called \textit{fibers}. Now if we want to obtain a multiset (which we imagine as a nonnegative integer valued function on $\Z_m$) such that the corresponding weighted sums of roots of unity is equal to zero, then we may take any linear combination of characteristic function of fibers or simply linear combination of fibers.

It was proved in \cite{deb, F1, Lam, schoen} that every multiset whose corresponding sum of roots of unities is equal to zero can be written as the linear combination of fibers. 
Notice that the same problem can be formulated for grid points contained in a suitable dimensional rectangle so we may change the sidelengths to any composite numbers. This was investigated by Coppersmith and Steinberger \cite{Stein}. They proved a lower bound for the size (sum of entries) of minimal sets corresponding the vanishing sums of roots of unities that do not contain a complete fiber. 

Coppersmith and Steinberger called the linear combination of fibers a cyclotomic array, although it was first introduced in \cite{steinberger}. They proved that any minimal cyclotomic array, which is not a coset, contains at least $(p-1)(q-1)+(r-1)$ points, where $p$ and $q$ are the sizes of the shortest sides of the paralellograms. 

Notice that in the case of $\F_p^3$ we obtain that the smallest set contained in $\F_p^3$ which is the linear combination of fibers is of size at least $p(p-1)$. 
Let us take two skew lines with weight $-1$ and let us establish bijection of the points of these two lines. Now we add lines determined by this bijection (these pair of points determine lines) with weight 1 to the sum of the first two lines. Thus we obtain a multiset of size $p(p-1)$. We verified that if these two lines are $\{ (i,0,0) \mid i \in \F_p \}$ and $\{ (0,1,i) \mid i \in \F_p \}$, then of the bijection is the natural one, then the resulting set of size $p(p-2)$ is actually a set not containing a complete line. The same fact was established for \textit{any bijection} by D\'avid Foris \cite{Foris}.% who is a \red{first year} student at ELTE. 

It was proved in \cite{BDMS, samthesis} that if a set satisfying the conditions of the cylinder set conjecture and contains a line, then it is indeed the union of $p$ parallel lines.  
Thus we raise here the following natural question. 
\begin{question}
What is the size of the smallest set in $\F_p^3$ that is the linear combination of (not necessarily parallel to the axis) lines?     
\end{question}

\section{Relation to Rédei's conjecture}\label{sec:7}

We begin with a conjecture of Rédei concerning additive decompositions in vector spaces over prime fields.
\begin{conj}[Rédei \cite{redei}]
If $A,B \subseteq \mathbb{F}_p^3$ satisfy $A+B = \mathbb{F}_p^3$ and $0\in A, 0\in B$, then $\langle A\rangle \ne \mathbb{F}_p^3$ or $\langle B\rangle \ne \mathbb{F}_p^3$.
\end{conj}

Motivated by this conjecture, we explain how a combination of two geometric–combinatorial statements imply Rédei’s conjecture. The first statement is the strong cylindrical set conjecture, and the second is an assumption on the number of directions determined by a set of cardinality 
$p$ in $\F_p^3$.	
The connection between these ideas appears in the Phd thesis of Sam Mattheus \cite{samthesis}.% in smaller details with a large amount of other details.  \footnote{We learned the existence of the connection from him during his visit to the Rényi Institute\url{https://erdoscenter.renyi.hu/articles/fourier-analysis-and-additive-problems-2024-spring-jan-june}}.
We present below a self-contained and compact argument. Connecting to this renowned conjecture of Rédei, we can say the following.
If the strong cylinder conjecture holds, then in combination with the assumption that every $T \subseteq \mathbb{F}_p^3$ of size $|T|=p$, which is not contained in a plane, determines at least $p+2$ directions, it would imply a conjecture of Rédei concerning additive decompositions in $\mathbb{F}_p^3$.

%In order to have a complete picture regarding the connection of the cylindrical set conjecture and Rédei's tiling conjecture on $\F_p^3$ we present an argument showing that a combination of the strong cylindrical set conjecture and knowledge on the number of directions determined by a set of size $p$ implies Rédei's conjecture. \todo{Le van ez barhol is irva? Mattheus teziseben ctrl+f alapjan nincs, szoval meg akar hasznos is lehet az irasban torteno rogzitese}

%Notice that we learned about this connection from Sam Mattheus when he visited us at the Rényi Institute\footnote{programot megnevezni} although the following is our argument. 

We present the argument in 4 steps. 

\textbf{Step 1: Preliminaries on tilings in finite abelian groups.} Recall that a pair of subsets $A,B$ of a finite abelian group $G$ forms a \emph{tile} if
\[
A \oplus B = G ,
\]
i.e.,\ the representation $g=a+b$ with $a\in A$, $b\in B$ is unique for every $g\in G$. 
It is a classical result observed by Sands \cite{Sands} that this holds if and only if
\[
|A||B| = |G|
\qquad\text{and}\qquad
(A-A)\cap (B-B)=\{0\}.
\]
A second standard fact is that tilings are preserved under certain multiplications: if $A\oplus B=G$ and $\gcd(r,|G|)=1$, then
\[
rA\oplus B = G,
\]
where $rA=\{ra : a\in A\}$ and $ra=\underbrace{a+a+ \ldots +a}_{r \text{ copies}}$.

Specializing to the case $G=\mathbb{F}_p^3$, these two observations imply that the sets of directions determined by $A$ and by $B$ are disjoint. 
We refer to this as the \emph{direction exclusion principle}.

\bigskip

\textbf{Step 2: Directions determined by small sets in $\mathbb{F}_p^3$.}

The second statement we use is about subsets of $\F_p^3$ of cardinality $p$. We emphasize that it is not yet known whether the statement is true or false, but we will show how it can be used to prove Rédei's conjecture if it proves to be true.  %\todo{ez a chatgpt irta, de ez igaz, es inkabb hol van leirva}
\begin{conj}[\cite{samthesis}, Conjecture 6.2.7]
\textbf{Directional dichotomy for sets of size $p$.}
Let $A\subseteq\mathbb{F}_p^3$ with $|A|=p$. Then either  
\begin{enumerate}
\item[(i)] $A$ is contained in a plane, or  
\item[(ii)] $A$ determines at least $p+1$ (possibly at least $p+2$) directions.  
\end{enumerate}
\end{conj}

This will allow us to reduce Rédei’s conjecture to the strong cylindrical set conjecture.

\bigskip

\textbf{Step 3: From direction exclusion to cylindrical structure.}

Assume now that
\[
A\oplus B = \mathbb{F}_p^3,
\qquad |A|=p,\; |B|=p^2,
\]
and that $A$ determines at least $p$ directions.

By the direction exclusion principle, $B$ determines \emph{no} such directions; hence $B$ determines strictly fewer than $p^2+2$ directions. 
By a result in \cite{BBBSS}---presented there with a proof using Rédei polynomials and said to be folklore---any subset of $\mathbb{F}_p^3$ of size $p^2$ that determines fewer than $p^2+2$ directions must satisfy the hypotheses of the strong cylindrical set conjecture. 
Consequently, assuming the strong cylindrical set conjecture, $B$ must be a union of $p$ parallel lines. 
After an affine change of coordinates, we may suppose these lines are vertical.

Let $\bar A,\bar B\subseteq\mathbb{F}_p^2$ denote the support of the vertical projections of $A$ and $B$, respectively. 
Because the lines in $B$ are vertical and disjoint, $|\bar B|=p$. 
The direction-exclusion principle implies that  $\bar A$ is a set, and so $|\bar A|=p$ as well.

Since projection is a group homomorphism and being a multitile is preserved under homomorphisms, we have
\[
\bar A \oplus \bar B = \mathbb{F}_p^2 .
\]
\textbf{Step 4: Application of Rédei’s theorem in the plane.}
Rédei’s classical theorem in $\mathbb{F}_p^2$ states that if two sets of cardinality $p$ are not contained in lines, then each of them determines at least $\frac{p+3}{2}$ directions. 
But then the direction exclusion principle in $\mathbb{F}_p^2$ would forbid tiling. 
Therefore, in any tiling $\bar A\oplus\bar B=\mathbb{F}_p^2$, at least one of the sets $\bar A,\bar B$ must be contained in a line.

If $\bar A$ is contained in a line, then $A$ is contained in a plane in $\mathbb{F}_p^3$.  
If $\bar B$ is contained in a line, then $B$ is contained in a plane in $\mathbb{F}_p^3$.

In either case, one of $A,B$ is contained in a proper subgroup of $\mathbb{F}_p^3$, proving Rédei’s conjecture under the combined assumptions (strong cylindrical set conjecture together with a lower bound on directions determined by sets of size $p$).

%\textcolor{red}{Gondolatmenet: \begin{itemize}
% \item If $T\bigoplus S= F_p^3$, then $rT\bigoplus S=F_p^3$ $D(T)\cap D(S)=\emptyset$.
% \item Ha $|T|=p$ és nem egyenesen van $\to$ S cilider.
% \item 2 d-ben S metszete +T vetülete egy tile, így S vagy T a vetületben egy egyenes.  
% \item Ha T vetülete egyenes, akkor T síkban van.
% \item Ha S metszete egy egyenes, akkor S egy sík. 
%\end{itemize}}

%\todo[inline]{KG: cilinder typusú halmazokkal foglalkoztunk; mi az ami elegendő lenne ahhoz hogy befejezzük a bizonyítást (súlyösszeg csökkentés plusz: ha kicsi a súlyösszeg, akkor befejezhető) \\ NZ: - erről online azt beszéltük, hogy talán hagyjuk inkább. Szerintem lehet hogy a Réde kapcs mégis ide kéne az introból, mert az intro elég terjedelmes lett, és azt a kapcsolatot valóban érdemes rendesen leírni.\\
%KG: Rendben. Ide teszem be és rendesen kiírom, ami ott csak egy listában van összegezve.}

%\todo[inline]{Coppersmith et al típusú eredményekkel való kapcsolat}

%\todo[inline]{Példák: Cilinder halmaz közel minimális mérettel; Multihalmaz; Más? }

%\begin{conj}[Rédei \cite{redei}] If $A+B=\F_p^3$, with $0\in A, 0\in B$, then $\langle A\rangle \neq \F^3$ or  $\langle B\rangle \neq \F^3$
    
%\end{conj}

\end{document}